\documentclass[12pt]{article}
\usepackage{amssymb}
\usepackage{amsmath}
\usepackage{listings}
\usepackage{color}
\numberwithin{equation}{section}

\newtheorem{theorem}{Theorem}[section]
\newtheorem{corollary}{Corollary}[section]
\newtheorem{remark}{Remark}[section]

\newtheorem{example}{Example}[section]
\newenvironment{proof} 
{\par\noindent{\bf Proof.}} 
{\hfill$\scriptstyle \Box$}

\title{ On criteria for periodic wavelet frame\footnote{The research is supported by the Russian Science Foundation grant No. 23-11-00178, https://rscf.ru/project/23-11-00178/.}
}
\author{Anastassia Gorsanova\footnote{St. Petersburg State University, 7-9 Universitetskaya nab.,
St. Petersburg 199034, Russia}, 
 Elena Lebedeva\footnote{St. Petersburg State University, 7-9 Universitetskaya nab.,
St. Petersburg 199034, Russia}
}
\date{
an.gor275@gmail.com, ealebedeva2004@gmail.com
}

\begin{document}
\maketitle

\begin{abstract} 
 We provide constructive necessary and sufficient conditions for a family of periodic wavelets to be a Parseval wavelet frame. The criterion generalizes unitary and oblique extension principles. The case of one wavelet generator and refinable functions being trigonometric polynomials is discussed in details. 
\end{abstract}



\section{Introduction}
Let $\mathbb{N}$, $\mathbb{Z}$, $\mathbb{Z}_+$ and $\mathbb{R}$ be the set of natural numbers, integers, non-negative integers and real numbers respectively.  By $L_p,$ $1\le p<\infty$  we denote the Lebesgue space of all 1-periodic complex-valued functions $f$ such that $|f|^p$ is integrable. By $L_{\infty}$ the space of essentially bounded measurable  1-periodic complex-valued functions is denoted.  The term $l_p$ is reserved for the corresponding sequence spaces.
Notions $\langle\cdot,\cdot\rangle$ and $\|\cdot\|_2$ are used for denoting inner products and norms both in spaces $L_2$ and $l_2$, that is,
$$
\langle f,g\rangle=\int\limits^1_0f(x)\overline{g(x)}dx\quad \text{for}\quad f,g\in L_2\,,$$
and
$$\langle a,b\rangle=\sum\limits_{k\in\mathbb{Z}} a(k)\overline{b(k)}\quad \text{for}\quad a,b\in l_2\,.$$
Corresponding norms are determined by $\|\cdot\|_2=\langle\cdot,\cdot\rangle^{1/2}$.
The $n$-th Fourier coefficient of $f\in L_p$ is denoted by
 $\widehat{f}(n)=\langle f,e^{2\pi i n\cdot}\rangle$, $n\in\mathbb{Z}$.

Let $S^n_j$, $n,\, j\in\mathbb{Z}_+$, be a $2^{-j}n$-shift operator on $L_2$ given by $S^n_jf := f(\cdot+2^{-j}n)$\,. For 1-periodic functions it is enough to consider $2^{-j}n$-shifts only for $n\in R_j$, where
$$R_j=\left\{-2^{j-1}+1, -2^{j-1}+2, \ldots, 2^{j-1}\right\}\,.$$

For all $j\in\mathbb{Z}_+$ consider positive integers $\rho_j$ and functions 
$\psi^m_j \in L_2,\ m=1, 2, \ldots, \rho_j$. 
 The collection of functions
$$
\Psi =\left\{S^k_j\psi^m_j : j\in\mathbb{Z}_+, m=1, 2, \ldots, \rho_j, k\in\mathcal{R}_j\right\}
$$ 
is referred to as a periodic wavelet system.
It is said, that the system 
$
\Psi
$ 
forms a Parseval wavelet frame in $L_2$ if
$$
\|f\|^2_2 
=
\sum\limits_{j=0}^{\infty}\sum\limits_{m=1}^{\rho_j}\sum\limits_{k\in\mathcal{R}_j}|\langle f, S^k_j\psi^m_j\rangle|^2,
\quad \forall\,f\in L_2\,.
$$
Sometimes (see, for example,  \cite{ghs},  \cite{gt}), wavelet system is written in the form 
$
\{\varphi_0\}\cup\left\{S^k_j\psi^m_j : j\in\mathbb{Z}_+, m=1, 2, \ldots, \rho_j, k\in\mathcal{R}_j\right\},
$
where $\varphi_0$ is a scaling function of the first periodic mulitresolution analysis (MRA)  set $V_0$. In this paper  we study wavelet systems of both types that associate and do not associate with MRA, so for convenience we denote all functions without shifts uniformly as $\psi^{m}_0,$
$m=1,\dots, \rho_0.$ If it is necessary one can assume that $\varphi_0 = \psi^1_0.$  

The aim of the paper is to provide a constructive criterion for the periodic wavelet system to be a Parseval wavelet frame. 
we provide periodic counterparts of general characterization (Theorem \ref{Th1}) and of criterion in terms of fundamental functions (Theorem \ref{Th2}). Theorem \ref{Th3} and Corollary \ref{Cor1} give a particular case of new constructive criterion. Based on this special case, we present the main result in Theorem \ref{Th4}. Some practical questions on how to apply the proposed criterion are discussed in Examples \ref{Ex1} and \ref{Ex2}.

For the non-periodic wavelet families the following general characterization is well-known (see~\cite{Lem},~\cite{Daub},~\cite{Grip},~\cite{NPS}):

\textbf{Theorem A.}
\textit{Suppose $\psi \in L_2(\mathbb{R})$. The family
     $ 
2^{j/2} \psi\left(2^j \cdot +k\right), j,k\in\mathbb{Z},
$ 
forms a Parseval frame in $L_2(\mathbb{R})$ iff 
$$
\sum_{j \in \mathbb{Z}}
\left|\widehat{\psi}(2^j \cdot)\right|^2 =1 \mbox{ a.e.},
$$
$$
\sum_{j \in \mathbb{Z}_+}
\widehat{\psi}(2^j \cdot) \overline{\widehat{\psi}\left(2^j(\cdot +k)\right)} =0 \mbox{ a.e. for all odd } k.
$$
}

This result is a powerful theoretical tool, however it is unclear how to use it practically to design wavelets. On the base of this criterion the following theorem is proved in  
 ~\cite{RShUEP}:

\textbf{Theorem B.}
\textit{ Suppose a function $\psi_0\in L_2(\mathbb{R})$ satisfies a scaling equation 
$$
\widehat{\psi_0}(\omega)=\tau_{0}(\omega/2)\widehat{\psi_{0}}(\omega/2),
$$
$\widehat{\psi_0}$ is continuous at $0,$ $\widehat{\psi_0}(0)=1.$
	Wavelet functions $\psi_i\in L_2(\mathbb{R})$, $i=1, \ldots, r$ are defined as
$$
\widehat{\psi_i}(\omega)=\tau_{i}(\omega/2)\widehat{\psi_{0}}(\omega/2),
$$
$\tau_{i} \in L_{\infty}$, $i=0, \ldots, r$. 
The fundamental function of mulitresolution analysis is defined as
$$
\Theta(\omega) := 
\sum_{j=0}^{\infty} \sum_{i=1}^r 
\left|\tau_i\left(2^j \omega \right)\right|^2
 \prod_{m=0}^{j-1}
\left|\tau_0\left(2^m \omega \right)\right|^2. 
$$
Then
$
\left\{2^{j/2}\psi_i\left(2^j x +k\right) : j,k\in\mathbb{Z}, i=1, \ldots, r \right\}
$ 
forms a Parseval frame in $L_2(\mathbb{R})$ iff 
$$
\lim_{j\to -\infty}\Theta(2^j\omega) = 1,
$$
$$
\Theta(2 \omega) \tau_0(\omega) \overline{\tau_0(\omega+1/2)}+
\sum_{i=1}^r \tau_i(\omega) \overline{\tau_i(\omega+1/2)} =0 \mbox{ for  a.e. } \omega.
$$}
The authors note that ``in practice it may be hard to select'' wavelet masks $\tau_i,$ $i=1, \ldots, r$, they are given implicitly via $\Theta$.  
In Corollary~6.7  ~\cite{RShUEP} very important special case 
$\Theta \equiv 1$
is introduced and discussed for the first time. It has become the standard way universally applied in non-periodic and periodic setting known as the celebrated unitary extension principle (UEP).  However, it  gives only sufficient conditions for a wavelet system to be a Parseval frame. 

The main contribution of the current paper to the topic is a constructive, explicit criterion for a Parseval wavelet frame in periodic setting. In the next section we provide periodic counterparts of general characterization (Theorem \ref{Th1}) and of criterion in terms of fundamental functions (Theorem \ref{Th2}). Theorem \ref{Th3} and Corollary \ref{Cor1} give a particular case of new constructive criterion. Based on this special case, we present the main result in Theorem \ref{Th4}. Some practical questions on how to apply the proposed criterion are discussed in Examples \ref{Ex1} and \ref{Ex2}.

\section{Criteria for wavelet frames}
We start with a periodic counterpart of Theorem A, a general characterization of a Parseval wavelet frame in terms of Fourier coefficients 
\begin{theorem}
\label{Th1}
The family 
$
\left\{S^k_j\psi^m_j : j\in\mathbb{Z}_+, m=1, 2, \ldots, \rho_j, k\in\mathcal{R}_j\right\}
$ 
forms a Parseval wavelet frame in $L_2$ iff 

 \begin{equation}
 \label{1th1}
\sum_{j=0}^{\infty} \sum_{m=1}^{\rho_j} 2^j \left|\widehat{\psi^m_{j}}(n)\right|^2 = 1 \ \mbox{ for all } n \in \mathbb{Z},
\end{equation}
 \begin{equation}
 \label{2th1}
\sum_{q=0}^{j} 2^{q} \sum_{m=1}^{\rho_{q}} \overline{\widehat{\psi^m_{q}}(n)} \widehat{\psi^m_{q}}(2^j k + n)  = 0
\end{equation}
for all odd $k,$ $n \in \mathbb{Z},$ $j \in \mathbb{Z}_+$.     
\end{theorem}

\begin{proof}
Let $f$ be a trigonometric polynomial. 
    To check sufficiency we start with a standard trick in wavelet setting 
   \begin{equation*}
    \begin{split}
    &
\sum\limits_{k\in\mathcal{R}_j}|\langle f, S^k_j\psi^m_j\rangle|^2 
 = 
 \sum\limits_{k\in\mathcal{R}_j}\left|\int_0^1 f(x) \overline{S^k_j\psi^m_j}(x) \, dx \right|^2 
    = 
    \sum\limits_{k\in\mathcal{R}_j}\left|\sum_{n\in \mathbb{Z}} \widehat{f}(n)\overline{\widehat{\psi^m_j}(n)} e^{-2 \pi i 2^{-j} k n} \right|^2 
    \\
   &
   = \sum\limits_{k\in\mathcal{R}_j}\left|\sum\limits_{n\in\mathcal{R}_j}\sum_{q\in \mathbb{Z}} \widehat{f}(2^j q + n)\overline{\widehat{\psi^m_j}(2^j q + n)} e^{-2 \pi i 2^{-j} k(2^jq+n)} \right|^2
    \\
    &
    = \sum\limits_{k\in\mathcal{R}_j}\left| \sum\limits_{n\in\mathcal{R}_j} h^m_j(n) e^{-2 \pi i 2^{-j} k n}\right|^2,
    \end{split}   
   \end{equation*} 
where
$
h^m_j(n):=\sum_{q\in \mathbb{Z}} \widehat{f}(2^j q + n)
\overline{\widehat{\psi^m_j}(2^j q + n)}.
$
 By the Plancherel theorem for DFT, we get
 \begin{equation*}
 \begin{split}
  &   \sum\limits_{k\in\mathcal{R}_j}|\langle f, S^k_j\psi^m_j\rangle|^2
 = 
  2^j \sum\limits_{n\in\mathcal{R}_j} \left|h^m_j(n) \right|^2
  =
 2^j \sum\limits_{n\in\mathcal{R}_j}
 \left(
 \sum_{q\in \mathbb{Z}} \left|\widehat{f}(2^j q + n)
\overline{\widehat{\psi^m_j}(2^j q + n)}\right|^2 
\right.
\\
&
\left.+
\sum_{q\neq q'} \widehat{f}(2^j q + n) \overline{\widehat{f}(2^j q' + n)} 
 \overline{\widehat{\psi^m_j}(2^j q + n)} \widehat{\psi^m_j}(2^j q' + n)
 \right)
\\
 &
 =
 2^j \sum_{n\in \mathbb{Z}} \left|\widehat{f}(n) \overline{\widehat{\psi^m_j}( n)}   \right|^2
 +
 2^j \sum_{n\in \mathbb{Z}}  \sum_{q \neq 0}
 \widehat{f}(n) \overline{\widehat{f}(2^j q + n)} \,
 \overline{\widehat{\psi^m_j}( n)} \widehat{\psi^m_j}(2^j q + n). 
 \end{split}  
 \end{equation*}
 
 Therefore,
 \begin{equation}
 \label{3th1}
\begin{split}
    &
\sum\limits_{j=0}^{\infty}\sum\limits_{m=1}^{\rho_j}\sum\limits_{k\in\mathcal{R}_j}|\langle f, S^k_j\psi^m_j\rangle|^2 = 
 \sum_{n\in \mathbb{Z}}  \left|\widehat{f}(n) \right|^2
 \sum\limits_{j=0}^{\infty}\sum\limits_{m=1}^{\rho_j} 2^j \left|\widehat{\psi^m_{j}}(n)\right|^2
 \\
 &
 +  \sum\limits_{j=0}^{\infty}\sum\limits_{m=1}^{\rho_j} 2^j \sum_{n\in \mathbb{Z}}  \sum_{q \neq 0}
 \widehat{f}(n) \overline{\widehat{f}(2^j q + n)}  \,
 \overline{\widehat{\psi^m_j}( n)} \widehat{\psi^m_j}(2^j q + n).
 \end{split}
 \end{equation}
 For the second term replacing $q=k 2^p,$ where $k$ is odd, by the Fubini theorem and changing of summation index $j+p \to j$  we finally obtain
 \begin{equation}
 \label{4th1}
\begin{split}
    &
  \sum_{k \mbox{ {\footnotesize{odd}}}} \sum_{n\in \mathbb{Z}} \sum\limits_{j=0}^{\infty} \sum\limits_{m=1}^{\rho_j}   \sum\limits_{p=0}^{\infty} 2^j  
  \widehat{f}(n) \overline{\widehat{f}(2^{j+p} k + n)}  \,
 \overline{\widehat{\psi^m_j}( n)} \widehat{\psi^m_j}(2^{j+p} k + n)
 \\
 &
 = \sum_{k \mbox{ {\footnotesize{odd}}}} \sum_{n\in \mathbb{Z}} \sum\limits_{j=0}^{\infty}
 2^j \widehat{f}(n) \overline{\widehat{f}(2^{j} k + n)}  
 \sum\limits_{p=0}^{j} 2^{-p}  \sum\limits_{m=1}^{\rho_{j-p}}  
   \overline{\widehat{\psi^m_{j-p}}( n)} \widehat{\psi^m_{j-p}}(2^{j} k + n).
 \end{split}
 \end{equation}
 We may use the Fubini theorem. It can be proved in the same way as in   ~\cite{Grip}.
 The set of trigonometric polynomials is everywhere  dense  in $L_2$.  
 Thus, sufficiency is proved.

 To check the necessity we fix $n_0\in\mathbb{Z}$ and set $\widehat{f}(n) = \delta_{n,n_0}.$ 
 Then equality (\ref{3th1}) takes the form
 $$
\sum\limits_{j=0}^{\infty}\sum\limits_{m=1}^{\rho_j}\sum\limits_{k\in\mathcal{R}_j}|\langle f, S^k_j\psi^m_j\rangle|^2 = 
  \sum\limits_{j=0}^{\infty}\sum\limits_{m=1}^{\rho_j} 2^j \left|\widehat{\psi^m_{j}}(n_0)\right|^2.
 $$
By assumption, the left-hand side is equal to $\|f\|^2 =1.$ So,  equality (\ref{1th1}) is fulfilled.
We check  (\ref{2th1}). Substituting (\ref{1th1}) into (\ref{3th1}) and using (\ref{4th1}) by definition of the Parseval frame we get 
$$
  \sum_{k \,\mbox{{\footnotesize{odd}}}} \sum_{n\in \mathbb{Z}} \sum\limits_{j=0}^{\infty}
 2^j \widehat{f}(n) \overline{\widehat{f}(2^{j} k + n)}  
 \sum\limits_{p=0}^{j} 2^{-p}  \sum\limits_{m=1}^{\rho_{j-p}}  
   \overline{\widehat{\psi^m_{j-p}}( n)} \widehat{\psi^m_{j-p}}(2^{j} k + n)=0 \mbox{ for all } f\in L_2.
$$
We fix an odd integer $k_0$, $j_0\in\mathbb{Z}_+$, $n_0\in\mathbb{Z}$ and set for the first time
$\widehat{f}(n) =2^{-j_0/2}$ for $n=n_0$ and $n=2^{j_0}k_0 + n_0$, $\widehat{f}(n) =0$, otherwise.   
Substituting the function into the last equality we obtain
$$
\sum\limits_{p=0}^{j_0} 2^{-p}  \sum\limits_{m=1}^{\rho_{j_0-p}} \left( 
   \overline{\widehat{\psi^m_{j_0-p}}( n_0)} \widehat{\psi^m_{j_0-p}}(2^{j_0} k_0 + n_0)
   +
   \overline{\widehat{\psi^m_{j_0-p}}(2^{j_0} k_0 + n_0)} \widehat{\psi^m_{j_0-p}}(n_0)\right)
   =0
$$
For the second time, we set 
$\widehat{f}(n) =i 2^{-j_0/2}$ for $n=n_0$, $\widehat{f}(n) =2^{-j_0/2}$ for $n=2^{j_0}k_0 + n_0$, $\widehat{f}(n) =0$, otherwise. As a result, we get 
$$
i\sum\limits_{p=0}^{j_0} 2^{-p}  \sum\limits_{m=1}^{\rho_{j_0-p}} \left(  
   \overline{\widehat{\psi^m_{j_0-p}}( n_0)} \widehat{\psi^m_{j_0-p}}(2^{j_0} k_0 + n_0)
   -
   \overline{\widehat{\psi^m_{j_0-p}}(2^{j_0} k_0 + n_0)} \widehat{\psi^m_{j_0-p}}(n_0)\right)
   =0.
$$
The last two equalities prove  (\ref{2th1}). 

\end{proof}

In the following theorem we obtain necessary and sufficient conditions for a wavelet system to be a Parseval wavelet frame in terms of scaling masks, wavelet masks,  and fundamental coefficients. It is the periodic counterpart of Theorem B. 
Let $\mathcal{S}\left(2^j\right)$ be the space of all $2^j$-periodic sequences $\{c_j(k)\}_{k\in\mathbb{Z}}$ of complex numbers, i.e. $c_j(2^jp+k)=c_j(k)$ for all $p\in\mathbb{Z}$.
\begin{theorem}
\label{Th2}
  Suppose functions $\varphi_j\in L_2,\ j\in\mathbb{N}$, satisfy the refinement equation
\begin{equation}\label{refeq}
\widehat{\varphi_j}(n)=\sqrt{2}\,\widehat{a_{j+1}}(n)\widehat{\varphi_{j+1}}(n),\quad n\in\mathbb{Z}
\end{equation}
for some $\widehat{a_{j+1}}\in\mathcal{S}\left(2^{j+1}\right)$. 
For every $j\in\mathbb{Z}_+$ fix a positive integer $\rho_j$ and suppose  $\psi^m_j\in L_2$, $m=1, \ldots, \rho_j$ are defined by the wavelet equation
\begin{equation}\label{waveq}
\widehat{\psi^m_j}(n)=\sqrt{2}\,\widehat{b^m_{j+1}}(n)\widehat{\varphi_{j+1}}(n),\quad n\in\mathbb{Z}\,
\end{equation}
with some sequences $\widehat{b^m_{j+1}}\in\mathcal{S}\left(2^{j+1}\right),\ m=1, \ldots, \rho_{j}$.

Then the collection 
$\left\{S^k_j\psi^m_j : j\in\mathbb{Z}_+, m=1, 2, \ldots, \rho_j, k\in\mathcal{R}_j\right\}$ forms a Parseval wavelet frame in $L_2$
iff the following two conditions are fulfilled:
\begin{equation}\label{con1}
\lim\limits_{j\to\infty} 2^j |\widehat{\varphi_j}(n)|^2 \theta_j(n) = 1, \quad  n \in \mathbb{Z},
\end{equation}
\begin{equation}\label{con2}
 \theta_j(n)\overline{\widehat{a_{j+1}}(n)}\widehat{a_{j+1}}(n+2^j)+\sum\limits_{m=1}^{\rho_{j}}\overline{\widehat{b^m_{j+1}}(n)}\widehat{b^m_{j+1}}(n+2^j) = 0,
\quad n\in\mathbb{Z}, \  j\in\mathbb{Z}_+,
\end{equation}
unless $\widehat{\varphi_{j+1}}(n)=0$ or $\widehat{\varphi_{j+1}}(2^j k+n)=0$, 
where 
the sequence $\theta_j(n)\in S(2^j)$ is defined recursively
\begin{equation} 
\label{theta}
\begin{split}
&    \sum\limits_{m=1}^{\rho_{j}}\left|\widehat{b^m_{j+1}}(n)\right|^2 + 
    \theta_j(n) \left|\widehat{a_{j+1}}(n)\right|^2 = \theta_{j+1}(n),
\\
&
 \theta_1(n) = \sum\limits_{m=1}^{\rho_{0}}\left|\widehat{b^m_{1}}(n)\right|^2, 
 \qquad
 \theta_0(n) = 0.
\end{split}
\end{equation}
\end{theorem}

\begin{remark}
The sequences  $\widehat{a_{j}}(n),$ $\widehat{b^m_{j+1}}(n),$ $\theta_{j+1}(n),$ $n\in\mathbb{Z}$, are refereed to as scaling masks, wavelet masks,  and fundamental coefficients respectively.  
\end{remark}

\begin{proof}
Solving recursion (\ref{theta}), we get
\begin{equation}
\label{theta1}
\begin{split}
    &
\theta_q(n) = \sum\limits_{m=1}^{\rho_{q-1}}\left|\widehat{b^m_{q}}(n)\right|^2 + \sum\limits_{m=1}^{\rho_{q-2}}\left|\widehat{b^m_{q-1}}(n)\right|^2  \left|\widehat{a_{q}}(n)\right|^2+\dots
\\
&
    +\sum\limits_{m=1}^{\rho_{0}}\left|\widehat{b^m_{1}}(n)\right|^2
\prod_{p=2}^q\left|\widehat{a_{p}}(n)\right|^2.
\end{split}
\end{equation}
Therefore,
\begin{equation*}
    \begin{split}
&
\theta_q(n) 2^q |\widehat{\varphi_q}(n)|^2= \sum\limits_{m=1}^{\rho_{q-1}}\left|\widehat{b^m_{q}}(n)\right|^2 2^q |\widehat{\varphi_q}(n)|^2 + \sum\limits_{m=1}^{\rho_{q-2}}\left|\widehat{b^m_{q-1}}(n)\right|^2  \left|\widehat{a_{q}}(n)\right|^2 2^q |\widehat{\varphi_q}(n)|^2
\\
&
+\dots+\sum\limits_{m=1}^{\rho_{0}}\left|\widehat{b^m_{1}}(n)\right|^2
\prod_{p=2}^q\left|\widehat{a_{p}}(n)\right|^2 2^q |\widehat{\varphi_q}(n)|^2.        
    \end{split}
\end{equation*}
By (\ref{refeq}) and (\ref{waveq}), the right-hand side takes the form 
$$
\sum_{m=1}^{\rho_{q-1}} 2^{q-1} \left|\widehat{\psi^m_{q-1}}(n)\right|^2 + \sum_{m=1}^{\rho_{q-2}} 2^{q-2} \left|\widehat{\psi^m_{q-2}}(n)\right|^2 +\dots +
\sum_{m=1}^{\rho_{0}} 2^{0} \left|\widehat{\psi^m_{0}}(n)\right|^2.
$$
Therefore, we obtain the identity
$$
\sum_{j=0}^{q-1} \sum_{m=1}^{\rho_j} 2^j \left|\widehat{\psi^m_{j}}(n)\right|^2 
=
\theta_q(n) 2^q |\widehat{\varphi_q}(n)|^2.
$$
Now passing to the limit $q\to \infty$ we see that
condition   (\ref{con1}) is equivalent to  (\ref{1th1}).

To check the equivalence (\ref{2th1}) and (\ref{con2}), unless $\widehat{\varphi_{j+1}}(n)=0$ or $\widehat{\varphi_{j+1}}(2^j k+n)=0$, we start with the left-hand side of (\ref{2th1}). 
By (\ref{refeq}) and (\ref{waveq}) taking into account periodicity of $\widehat{a_{j}}(n)$ and $\widehat{b^m_{j}}(n)$ we get
\begin{equation*}
    \begin{split}
    &
\sum_{m=1}^{\rho_{j}} \overline{\widehat{\psi^m_{j}}(n)} \widehat{\psi^m_{j}}(2^j k + n)
= 
2\sum_{m=1}^{\rho_{j}}  \overline{\widehat{b^m_{j+1}}(n)} \widehat{b^m_{j+1}}(2^j k + n)
\overline{\widehat{\varphi_{j+1}}(n)}\widehat{\varphi_{j+1}}(2^j k + n) 
\\
&
=2\sum_{m=1}^{\rho_{j}}  \overline{\widehat{b^m_{j+1}}(n)} \widehat{b^m_{j+1}}(2^j  + n)
\overline{\widehat{\varphi_{j+1}}(n)}\widehat{\varphi_{j+1}}(2^j k + n).   
    \end{split}
\end{equation*}
\begin{equation*}
    \begin{split}
    &
\sum_{m=1}^{\rho_{j-1}} \overline{\widehat{\psi^m_{j-1}}(n)} \widehat{\psi^m_{j-1}}(2^j k + n)
= 
2\sum_{m=1}^{\rho_{j-1}}  \overline{\widehat{b^m_{j}}(n)} \widehat{b^m_{j}}(2^j k + n)
\overline{\widehat{\varphi_{j}}(n)}\widehat{\varphi_{j}}(2^j k + n) 
\\
&
=2^2\sum_{m=1}^{\rho_{j-1}}  \left|\widehat{b^m_{j}}(n)\right|^2 
\overline{\widehat{a_{j+1}}(n)}\,
\widehat{a_{j+1}}(2^j + n)
\overline{\widehat{\varphi_{j+1}}(n)}\widehat{\varphi_{j+1}}(2^j k + n). 
 \end{split}
\end{equation*}
Continue calculations we get for $j=0$
\begin{equation*}
    \begin{split}
    &
\sum_{m=1}^{\rho_{0}} \overline{\widehat{\psi^m_{0}}(n)} \widehat{\psi^m_{0}}(2^j k + n)
= 
2\sum_{m=1}^{\rho_{0}}  \overline{\widehat{b^m_{1}}(n)} \widehat{b^m_{1}}(2^j k + n)
\overline{\widehat{\varphi_{1}}(n)}\widehat{\varphi_{1}}(2^j k + n) 
\\
&
=2^{j+1}\sum_{m=1}^{\rho_{0}}  \left|\widehat{b^m_{1}}(n)\right|^2
\prod_{q=2}^j\left|\widehat{a_{q}}(n)\right|^2 
\overline{\widehat{a_{j+1}}(n)}\,
\widehat{a_{j+1}}(2^j + n)
\overline{\widehat{\varphi_{j+1}}(n)}\widehat{\varphi_{j+1}}(2^j k + n). 
 \end{split}
\end{equation*}
Substituting all these expressions into the left-hand side of (\ref{2th1}) we obtain 
\begin{equation*}
    \begin{split}
    &
\sum_{q=0}^{j} 2^{q} \sum_{m=1}^{\rho_{q}} \overline{\widehat{\psi^m_{q}}(n)} \widehat{\psi^m_{q}}(2^j k + n) 
= 
2^{j+1} 
\left(
\sum_{m=1}^{\rho_{j}}  \overline{\widehat{b^m_{j+1}}(n)} \widehat{b^m_{j+1}}(2^j  + n)
\right.
\\
&
\left.
+
\left(
\sum_{m=1}^{\rho_{j-1}}  \left|\widehat{b^m_{j}}(n)\right|^2
+\dots+
\sum_{m=1}^{\rho_{0}}  \left|\widehat{b^m_{1}}(n)\right|^2
\prod_{q=2}^j\left|\widehat{a_{q}}(n)\right|^2
\right)
\right.
\\
&
\left.
\times\overline{\widehat{a_{j+1}}(n)}\,
\widehat{a_{j+1}}(2^j + n)
\right)
\overline{\widehat{\varphi_{j+1}}(n)}\widehat{\varphi_{j+1}}(2^j k + n)
\end{split}
\end{equation*}
By (\ref{theta1}), the right-hand side of the last equality takes the form
$$
2^{j+1} 
\left(
\sum_{m=1}^{\rho_{j}}  \overline{\widehat{b^m_{j+1}}(n)} \widehat{b^m_{j+1}}(2^j  + n)
+
\theta_j(n)
\overline{\widehat{a_{j+1}}(n)}\,
\widehat{a_{j+1}}(2^j + n)
\right)
\overline{\widehat{\varphi_{j+1}}(n)}\widehat{\varphi_{j+1}}(2^j k + n).
$$
Thus, the equivalence (\ref{2th1}) and (\ref{con2}), unless $\widehat{\varphi_{j+1}}(n)=0$ or $\widehat{\varphi_{j+1}}(2^j k+n)=0$, follows.
\end{proof}

\begin{remark}
In the unitary extension principle, there is one more assumption on the refinable sequence  (see Theorem 2.1   \cite{ghs}) 
$$
\lim\limits_{q\to\infty} 2^q |\widehat{\varphi_q}(n)|^2 =1.
$$
It follows from Theorem \ref{Th2} that actually we do not need this condition to get a Parseval frame. However, if we add this assumption to the statement, then (\ref{con1}) should be replaced by
$$
\lim\limits_{q\to\infty} \theta_j(n) =1.
$$
\hfill $\scriptstyle \Box$
\end{remark}

Theorem \ref{Th2} is a periodic counterpart of Theorem 6.5   \cite{RShUEP} cited in the Introduction as Theorem B. The sequence $\theta_j(n)$ corresponds to the fundamental function of multiresolution  $\Theta$.   
Looking at these Theorems, the natural question arises. How to use the results?   Suppose refinable functions  are known. How to find wavelets? In other words,  the coefficients $\widehat{a_{j+1}}(n)$ are given. How to find $\widehat{b^m_{j+1}}(n)$?
For non-periodic case in Corollary 6.7   \cite{RShUEP} special case 
$\Theta \equiv 1$, known as the unitary extension principle,
is discussed for the first time. In  \cite{dhrs} a more flexible recipe for designing of frames is provided, namely the oblique extension principle, however it is proved that  the unitary and oblique extension principles are equivalent. Both principles give only sufficient conditions for a wavelet system to be a Parseval frame.
In periodic setting it turns out that it is possible give a constructive description of   $\theta_j(n)$ 
 and wavelet coefficients  $\widehat{b^m_{j+1}}(n)$.  
In the following two  Theorems and one Corollary we propose an explicit necessary and sufficient condition for a wavelet system to be a Parseval frame in terms 
of the scaling coefficients $\widehat{a_{j+1}}(n)$ and auxiliary explicitly defined coefficients  only. In Theorem \ref{Th3}, there are two additional assumptions    
``$\widehat{\varphi_j}(n) \neq 0$ for all $j\in \mathbb{N}$, $n\in\mathbb{Z}$''
and 
``$
\sum\limits_{m=1}^{\rho_{0}}\left|\widehat{b^m_{1}}(n)\right|^2  \neq 0.
$
''
We need these restrictions to provide the essential ingredient of the proof without distracting from secondary details. In fact, the proof of the general result in Theorem \ref{Th4} is reduced to this particular case. Corollary  \ref{Cor1} discusses the case of one wavelet generator $\rho_j=1.$

\begin{theorem}
\label{Th3}
    Suppose  functions $\varphi_j\in L_2,\ j\in\mathbb{N}$, satisfy the refinement 
    equation (\ref{refeq})
for some $\widehat{a_{j+1}}\in\mathcal{S}\left(2^{j+1}\right)$
and $\widehat{\varphi_j}(n) \neq 0$ for all $j\in \mathbb{N}$, $n\in\mathbb{Z}$.
For every $j\in\mathbb{Z}_+$ fix a positive integer $\rho_j$ and suppose  $\psi^m_j\in L_2$, $m=1, \ldots, \rho_j$ are defined by the wavelet equation
(\ref{waveq})
with some sequences $\widehat{b^m_{j+1}}\in\mathcal{S}\left(2^{j+1}\right),\ m=1, \ldots, \rho_{j}$. Suppose 
$
\sum\limits_{m=1}^{\rho_{0}}\left|\widehat{b^m_{1}}(n)\right|^2  \neq 0.
$

We introduce auxiliary coefficients   
$
\tilde{a}_{j}(n),\, \tilde{b}^m_{j}(n) \in \mathcal{S}\left(2^{j}\right),
j\ge 2,
$
\begin{equation}
\label{tildea}
\tilde{a}_{j}(n) = 
\left( \prod\limits_{r=1}^{\rho_{j-1}} \cos t^k_r \right)
e^{i t^{k}_{\rho_{j-1}+1}} \neq 0, 
\end{equation}
\begin{equation}
\label{tildeb}
\tilde{b}^m_{j}(n) = 
\left( \prod\limits_{r=m+1}^{\rho_{j-1}} \cos t^k_r \right) \sin t^k_m 
e^{i t^{k}_{\rho_{j-1}+1+m}},
\end{equation}
$k=0$ for $n=0,\dots,2^{j-1}-1$, 
$k=1$ for $n=2^{j-1},\dots, 2^{j}-1$, $j \in \mathbb{N},$ 
where the parameters 
$t^k_r =t^k_r(j,n),$ $r=1,\dots,  2\rho_{j-1}+1,$
satisfy the system
\begin{equation}
\label{sys2}
\begin{cases}
\sum\limits_{m=0}^{\rho_{j-1}} \cos \left(t^0_{\rho_{j-1}+1+m}-t^1_{\rho_{j-1}+1+m}\right)
\prod\limits_{k=0}^1\sin t^k_m 
\prod\limits_{r=m+1}^{\rho_{j-1}} \cos t^k_r = 0,
\\
\sum\limits_{m=0}^{\rho_{j-1}} \sin \left(t^0_{\rho_{j-1}+1+m}-t^1_{\rho_{j-1}+1+m}\right)
\prod\limits_{k=0}^1\sin t^k_m 
\prod\limits_{r=m+1}^{\rho_{j-1}} \cos t^k_r = 0,
\end{cases} 
\end{equation}
the factor 
$
\prod\limits_{r=\rho_{j-1}+1}^{\rho_{j-1}} \cos t^k_r 
$
is understood as $1$.

Then the collection 
$\left\{S^k_j\psi^m_j : j\in\mathbb{Z}_+, m=1, 2, \ldots, \rho_j, k\in\mathcal{R}_j\right\}$ forms a Parseval wavelet frame in $L_2$
iff the following conditions are fulfilled
\begin{equation}
\label{b1}
 \sum\limits_{m=1}^{\rho_{0}}\overline{\widehat{b^m_{1}}(0)}\widehat{b^m_{1}}(1) = 0,   
\end{equation}
\begin{equation}\label{infpr}
 \prod_{r=2}^{\infty}  \left|\tilde{a}_{r}(n)\right|^2 = 
2\sum\limits_{m=1}^{\rho_{0}}\left|\widehat{b^m_{1}}(n)\right|^2 \left|\widehat{\varphi_1}(n)\right|^2.   
\end{equation}
Moreover,  the wavelet coefficients are defined by 
\begin{equation}\label{bj}
\widehat{b^m_{j}}(n)
= 
\tilde{b}^m_{j}(n)
\sqrt{\sum\limits_{m=1}^{\rho_{0}}\left|\widehat{b^m_{1}}(n)\right|^2}
\prod_{r=2}^j  \frac{\left|\widehat{a_{r}}(n)\right|}{    \left|\tilde{a}_{r}(n)\right|}.
\end{equation}
\end{theorem}

\begin{proof}
    We need to check whether  equations (\ref{con1})-(\ref{theta}) are equivalent to equations (\ref{b1})-(\ref{bj}). First, we give a complete parametric solution of the system consisting of three equations, one copy of  (\ref{con2}) and two copies of (\ref{theta}), one for $n$, another for $n+ 2^{j-1}$
    \begin{equation}
        \label{sys_i}
    \begin{cases}
    \theta_{j-1}(n)\overline{\widehat{a_{j}}(n)}\widehat{a_{j}}(n+2^{j-1})+\sum\limits_{m=1}^{\rho_{j-1}}\overline{\widehat{b^m_{j}}(n)}\widehat{b^m_{j}}(n+2^{j-1}) = 0, 
    \\
    \sum\limits_{m=1}^{\rho_{j-1}}\left|\widehat{b^m_{j}}(n)\right|^2 + 
    \theta_{j-1}(n) \left|\widehat{a_{j}}(n)\right|^2 = \theta_{j}(n),
    \\
    \sum\limits_{m=1}^{\rho_{j-1}}\left|\widehat{b^m_{j}}(n+2^{j-1})\right|^2 + 
    \theta_{j-1}(n) \left|\widehat{a_{j}}(n+2^{j-1})\right|^2 = \theta_{j}(n+2^{j-1}).
\end{cases}    
    \end{equation}
    To this end we introduce auxiliary sequences 
\begin{equation}
\label{conthet}
    \tilde{a}_{j}(n) := \sqrt{\frac{\theta_{j-1}(n)}{\theta_{j}(n)}}\widehat{a_{j}}(n)
    \qquad
    \tilde{b}^m_{j}(n) := \sqrt{\frac{1}{\theta_{j}(n)}}\widehat{b^m_{j}}(n).
\end{equation}
The assumption $\widehat{\varphi_j}(n) \neq 0$ for all $j \in \mathbb{N}$, $n\in\mathbb{Z}$,
 and (\ref{refeq}) imply that  $\widehat{a_{j}}(n) \neq 0$ for all $j \ge 2$, $n\in\mathbb{Z}$. Then by the assumption 
 $
\sum\limits_{m=1}^{\rho_{0}}\left|\widehat{b^m_{1}}(n)\right|^2  \neq 0
$
and (\ref{theta}), we conclude that $\theta_{j}(n)\neq 0$ for all $j\in \mathbb{N}$. 
Therefore, $\tilde{a}_{j}(n)$ and $\tilde{b}^m_{j}(n)$ are well-defined and 
$\tilde{a}_{j}(n) \neq 0$ $j \ge 2$, $n\in\mathbb{Z}$.

With respect to new variables system (\ref{sys_i}) takes the form
\begin{equation}
\label{sys0}
\begin{cases}
    \overline{\tilde{a}_{j}}(n)\tilde{a}_{j}(n+2^{j-1})+\sum\limits_{m=1}^{\rho_{j-1}}\overline{\tilde{b}^m_{j}}(n)\tilde{b}^m_{j}(n+2^{j-1}) = 0,
    \\
     \sum\limits_{m=1}^{\rho_{j-1}}\left|\tilde{b}^m_{j}(n)\right|^2 + 
    \left|\tilde{a}_{j}(n)\right|^2 =1,
    \\
      \sum\limits_{m=1}^{\rho_{j-1}}\left|\tilde{b}^m_{j}(n+2^{j-1})\right|^2 + 
    \left|\tilde{a}_{j}(n+2^{j-1})\right|^2 =1.
\end{cases}    
\end{equation}
We denote  
\begin{equation*}
    \begin{split}
    &
\tilde{a}_{j}(n) =:x^0_0+i y^0_0, \qquad  
\tilde{a}_{j}(n+2^{j-1}) =:x^1_0+i y^1_0,
\\
&
\tilde{b}^m_{j}(n) =:x^0_m+i y^0_m, \qquad  
\tilde{b}^m_{j}(n+2^{j-1}) =:x^1_m+i y^1_m.        
    \end{split}
\end{equation*}
Then the last system becomes
\begin{equation}
\label{sys}
\begin{cases}
\sum\limits_{m=0}^{\rho_{j-1}}\left((x^0_m)^2+(y^0_m)^2\right)=1,\\
\sum\limits_{m=0}^{\rho_{j-1}}\left((x^1_m)^2+(y^1_m)^2\right)=1,\\ \sum\limits_{m=0}^{\rho_{j-1}}\left(x^0_m x^1_m + y^0_m y^1_m\right)=0,\\
\sum\limits_{m=0}^{\rho_{j-1}}\left(x^0_m y^1_m - y^0_m x^1_m\right)=0.\\
\end{cases}
\end{equation}
To satisfy the first two equations we use a kind of spherical coordinates: first, we explore ordinary spherical coordinates for $\sqrt{(x^k_m)^2+(y^k_m)^2}$, $k=0,1$, 
$$
\begin{cases}
(x^k_0)^2+(y^k_0)^2=\cos^2 t^k_{\rho_{j-1}} \cos^2 t^k_{\rho_{j-1}-1} \dots \cos^2 t^k_{2} \cos^2 t^k_{1},
\\
(x^k_1)^2+(y^k_1)^2= \cos^2 t^k_{\rho_{j-1}} \cos^2 t^k_{\rho_{j-1}-1} \dots \cos^2 t^k_{2} \sin^2 t^k_{1},\\
(x^k_2)^2+(y^k_2)^2= \cos^2 t^k_{\rho_{j-1}} \cos^2 t^k_{\rho_{j-1}-1} \dots \sin^2 t^k_{2},\\
(x^k_3)^2+(y^k_3)^2=\cos^2 t^k_{\rho_{j-1}} \cos^2 t^k_{\rho_{j-1}-1} \dots \sin^2 t^k_{3}, \\
\dots \\
(x^k_{\rho_{j-1}-1})^2+(y^k_{\rho_{j-1}-1})^2= \cos^2 t^k_{\rho_{j-1}} \sin^2 t^k_{\rho_{j-1}-1},\\
(x^k_{\rho_{j-1}})^2+(y^k_{\rho_{j-1}})^2= \sin^2 t^k_{\rho_{j-1}},\\
\end{cases}
k=0,1.
$$
then we parameterize the last expressions by polar coordinates 
$$
\begin{cases}
x^k_0=\cos t^k_{\rho_{j-1}} \cos t^k_{\rho_{j-1}-1} \dots \cos t^k_{2} \cos t^k_{1} \cos t^k_{\rho_{j-1}+1},
\\
y^k_0= \cos t^k_{\rho_{j-1}} \cos t^k_{\rho_{j-1}-1} \dots \cos t^k_{2} \cos t^k_{1} \sin t^k_{\rho_{j-1}+1},\\
x^k_1=  \cos t^k_{\rho_{j-1}} \cos t^k_{\rho_{j-1}-1} \dots \cos t^k_{2} \sin t^k_{1} \cos t^k_{\rho_{j-1}+2},\\
y^k_1= \cos t^k_{\rho_{j-1}} \cos t^k_{\rho_{j-1}-1} \dots \cos t^k_{2} \sin t^k_{1} \sin t^k_{\rho_{j-1}+2}, \\
\dots \\
x^k_{\rho_{j-1}}=  \sin t^k_{\rho_{j-1}} \cos t^k_{2\rho_{j-1}+1},\\
y^k_{\rho_{j-1}}= \sin t^k_{\rho_{j-1}} \sin t^k_{2\rho_{j-1}+1},\\
\end{cases}
k=0,1.    
$$
Then  the last two equations of (\ref{sys}) read as follows
$$
\begin{cases}
\sum\limits_{m=0}^{\rho_{j-1}} \cos \left(t^0_{\rho_{j-1}+1+m}-t^1_{\rho_{j-1}+1+m}\right)
\prod\limits_{k=0}^1\sin t^k_m 
\prod\limits_{r=m+1}^{\rho_{j-1}} \cos t^k_r = 0,
\\
\sum\limits_{m=0}^{\rho_{j-1}} \sin \left(t^0_{\rho_{j-1}+1+m}-t^1_{\rho_{j-1}+1+m}\right)
\prod\limits_{k=0}^1\sin t^k_m 
\prod\limits_{r=m+1}^{\rho_{j-1}} \cos t^k_r = 0,
\end{cases}
$$
where 
$
\prod\limits_{r=\rho_{j-1}+1}^{\rho_{j-1}} \cos t^k_r 
$
is understood as $1$. Thus, $\tilde{a}_{j}(n)$ and $\tilde{b}^m_{j}(n)$ are defined by 
(\ref{tildea}) and (\ref{tildeb}) respectively and the parameters $t^k_r$ are linked by system  (\ref{sys2}).

Since $\tilde{a}_{2}(n),\dots,\tilde{a}_{j}(n) \neq 0$, then with a solution of system (\ref{sys0}) in hands we obtain from (\ref{conthet})
\begin{equation}
\label{thetarec}
\theta_{j}(n)= \frac{\left|\widehat{a_{j}}(n)\right|^2}{    \left|\tilde{a}_{j}(n)\right|^2} \theta_{j-1}(n)    
\end{equation}
and
\begin{equation}
    \label{brec}
 \widehat{b^m_{j}}(n) =  \tilde{b}^m_{j}(n) \sqrt{\theta_{j}(n)} = 
 \tilde{b}^m_{j}(n)
 \frac{\left|\widehat{a_{j}}(n)\right|}{ \left|\tilde{a}_{j}(n)\right|}
 \sqrt{\theta_{j-1}(n)}.
\end{equation}
Solving the recursion, we get
\begin{equation}
\label{theta2}
 \theta_{j}(n)
=
\theta_{1}(n)
\prod_{r=2}^{j}  \frac{\left|\widehat{a_{r}}(n)\right|^2}{    \left|\tilde{a}_{r}(n)\right|^2} 
=
\sum\limits_{m=1}^{\rho_{0}}\left|\widehat{b^m_{1}}(n)\right|^2
\prod_{r=2}^{j}  \frac{\left|\widehat{a_{r}}(n)\right|^2}{    \left|\tilde{a}_{r}(n)\right|^2}.    
\end{equation}
By (\ref{brec}) and (\ref{theta2}), we get (\ref{bj})
$$
\widehat{b^m_{j}}(n)  = 
\tilde{b}^m_{j}(n)
\sqrt{\sum\limits_{m=1}^{\rho_{0}}\left|\widehat{b^m_{1}}(n)\right|^2}
\prod_{r=2}^{j}  \frac{\left|\widehat{a_{r}}(n)\right|}{    \left|\tilde{a}_{r}(n)\right|},
$$
where $\widehat{b^m_{1}}(n)$ ($2$-periodic sequence)  is chosen according to (\ref{con2}), that is
$$
 \sum\limits_{m=1}^{\rho_{0}}\overline{\widehat{b^m_{1}}(0)}\widehat{b^m_{1}}(1) = 0.   
$$
The last condition we need to satisfy is (\ref{con1}). By (\ref{refeq}), we get
$$
\widehat{\varphi_1}(n) 
= 2^{(j-1)/2}\widehat{\varphi_j}(n) \prod_{r=2}^{j}  \widehat{a_r}(n) .
$$
Together with (\ref{theta2}), we obtain 
$$
 2^j |\widehat{\varphi_j}(n)|^2 \theta_j(n) 
 = 
  2^j |\widehat{\varphi_j}(n)|^2 
\sum\limits_{m=1}^{\rho_{0}}\left|\widehat{b^m_{1}}(n)\right|^2
\prod_{r=2}^j  \frac{\left|\widehat{a_{r}}(n)\right|^2}{\left|\tilde{a}_{r}(n)\right|^2}
$$
$$
=2 \left|\widehat{\varphi_1}(n)\right|^2 
\sum\limits_{m=1}^{\rho_{0}}\left|\widehat{b^m_{1}}(n)\right|^2
\frac{1}{\prod_{r=2}^j  \left|\tilde{a}_{r}(n)\right|^2}.
$$
Thus, (\ref{con1}) is equivalent to  (\ref{infpr}).
\end{proof}

\begin{remark}
System (\ref{sys2}) can be solved with respect to two parameters by various ways. For example,     interpreting  (\ref{sys2})   as a linear system with respect to $\cos t^0_{\rho_{j-1}} \cos t^1_{\rho_{j-1}}$ 
and $\sin t^0_{\rho_{j-1}} \sin t^1_{\rho_{j-1}}$
$$
\begin{cases}
\cos t^0_{\rho_{j-1}} \cos t^1_{\rho_{j-1}} \sum\limits_{m=0}^{\rho_{j-1}-1} 
\cos \left(t^0_{\rho_{j-1}+1+m}-t^1_{\rho_{j-1}+1+m}\right)
\prod\limits_{k=0}^1\sin t^k_m 
\prod\limits_{r=m+1}^{\rho_{j-1}-1} \cos t^k_r 
\\
+
\sin t^0_{\rho_{j-1}} \sin t^1_{\rho_{j-1}} 
\cos \left(t^0_{2\rho_{j-1}+1}-t^1_{2\rho_{j-1}+1}\right)
= 0,
\\
\cos t^0_{\rho_{j-1}} \cos t^1_{\rho_{j-1}} \sum\limits_{m=0}^{\rho_{j-1}-1} 
\sin \left(t^0_{\rho_{j-1}+1+m}-t^1_{\rho_{j-1}+1+m}\right)
\prod\limits_{k=0}^1\sin t^k_m 
\prod\limits_{r=m+1}^{\rho_{j-1}-1} \cos t^k_r 
\\
+
\sin t^0_{\rho_{j-1}} \sin t^1_{\rho_{j-1}} 
\sin \left(t^0_{2\rho_{j-1}+1}-t^1_{2\rho_{j-1}+1}\right)
= 0,
\end{cases}
$$
 we conclude that if
$$
\sum\limits_{m=0}^{\rho_{j-1}-1} \sin \left(t^0_{\rho_{j-1}+1+m}-t^1_{\rho_{j-1}+1+m}-\left(t^0_{2\rho_{j-1}+1}-t^1_{2\rho_{j-1}+1}\right)\right)
\prod\limits_{k=0}^1\sin t^k_m 
\prod\limits_{r=m+1}^{\rho_{j-1}-1} \cos t^k_r \neq 0,
$$
then $\cos t^0_{\rho_{j-1}} = \sin t^1_{\rho_{j-1}} =0$ or $\cos t^1_{\rho_{j-1}} = \sin t^0_{\rho_{j-1}} =0,$ the last equations contradict the condition $\tilde{a}_{j}(n) \neq 0$.
So 
we conclude that the equations of the system are linearly dependent. We obtain the equivalent system if  we keep one of the equations and accompany it by the zero determinant of system  (\ref{sys2}) 
$$
    \begin{cases}
\cos t^0_{\rho_{j-1}} \cos t^1_{\rho_{j-1}} \sum\limits_{m=0}^{\rho_{j-1}-1} 
\cos \left(t^0_{\rho_{j-1}+1+m}-t^1_{\rho_{j-1}+1+m}\right)
\prod\limits_{k=0}^1\sin t^k_m 
\prod\limits_{r=m+1}^{\rho_{j-1}-1} \cos t^k_r 
\\
+
\sin t^0_{\rho_{j-1}} \sin t^1_{\rho_{j-1}} 
\cos \left(t^0_{2\rho_{j-1}+1}-t^1_{2\rho_{j-1}+1}\right)
= 0,
\\
\sum\limits_{m=0}^{\rho_{j-1}-1} \sin \left(t^0_{\rho_{j-1}+1+m}-t^1_{\rho_{j-1}+1+m}-\left(t^0_{2\rho_{j-1}+1}-t^1_{2\rho_{j-1}+1}\right)\right)
\prod\limits_{k=0}^1\sin t^k_m 
\prod\limits_{r=m+1}^{\rho_{j-1}-1} \cos t^k_r = 0.
\end{cases}
$$
The resulting system can be solved with respect to two variables, for example, if $t^1_{\rho_{j}-1}\neq \pi n,$ $t^0_{2\rho_{j-1}}-t^1_{2\rho_{j-1}}-\left(t^0_{2\rho_{j-1}+1}-t^1_{2\rho_{j-1}+1}\right)\neq \pi n$, $t^1_{\rho_{j}}\neq \pi n,$ $t^0_{2\rho_{j-1}+1}-t^1_{2\rho_{j-1}+1}\neq \pi/2 + \pi n$,  $n\in \mathbb{Z}$, then 
we can find
 $\tan t^0_{\rho_{j-1}-1}$ from the second equation
 \begin{equation*}
     \begin{split}
   &
     \tan t^0_{\rho_{j-1}-1} = - \cot t^1_{\rho_{j-1}-1}
\\
&
     \times \frac{
\sum\limits_{m=0}^{\rho_{j-1}-2} \sin \left(t^0_{\rho_{j-1}+1+m}-t^1_{\rho_{j-1}+1+m}-\left(t^0_{2\rho_{j-1}+1}-t^1_{2\rho_{j-1}+1}\right)\right)
\prod\limits_{k=0}^1\sin t^k_m 
\prod\limits_{r=m+1}^{\rho_{j-1}-1} \cos t^k_r}{\sin \left(t^0_{2\rho_{j-1}}-t^1_{2\rho_{j-1}}-\left(t^0_{2\rho_{j-1}+1}-t^1_{2\rho_{j-1}+1}\right)\right)}
     \end{split}
 \end{equation*}
 and substituting it into  the first one we can express 
$\tan t^0_{\rho_{j-1}}$ 
$$
\tan t^0_{\rho_{j-1}}  = -\cot t^1_{\rho_{j-1}}  \frac{ \sum\limits_{m=0}^{\rho_{j-1}-1} 
\cos \left(t^0_{\rho_{j-1}+1+m}-t^1_{\rho_{j-1}+1+m}\right)
\prod\limits_{k=0}^1\sin t^k_m 
\prod\limits_{r=m+1}^{\rho_{j-1}-1} \cos t^k_r }{\cos \left(t^0_{2\rho_{j-1}+1}-t^1_{2\rho_{j-1}+1}\right)}.
$$
A bit another solution of system (\ref{sys2}) is provided in the following Corollary.
\hfill $\scriptstyle \Box$
\end{remark}

We look closer at a very important particular case of one wavelet generator $\rho_j=1.$
As a result we obtain the following 

\begin{corollary}
\label{Cor1}

  Suppose functions $\varphi_j\in L_2,\ j\in\mathbb{N}$, satisfy the refinement equation
(\ref{refeq})
for some $\widehat{a_{j+1}}\in\mathcal{S}\left(2^{j+1}\right)$ and $\widehat{\varphi_j}(n) \neq 0$ for all $j\in \mathbb{N}$, $n\in\mathbb{Z}$.
Let  $\psi_j\in L_2$ $j\in\mathbb{N}$, be defined by the wavelet equation
$$
\widehat{\psi_j}(n)=\sqrt{2}\,\widehat{b_{j+1}}(n)\widehat{\varphi_{j+1}}(n),\quad n\in\mathbb{Z}\,
$$
with some sequences $\widehat{b_{j+1}}\in\mathcal{S}\left(2^{j+1}\right)$.
Let $m=1,2,$ 
$$
\widehat{\psi^m_0}(n)=\sqrt{2}\,\widehat{b^m_{1}}(n)\widehat{\varphi_{1}}(n),\quad n\in\mathbb{Z}\,
$$
with some sequences $\widehat{b^m_{1}}\in\mathcal{S}\left(2\right)$.
 Suppose 
$
\sum\limits_{m=1}^{2}\left|\widehat{b^m_{1}}(n)\right|^2 \neq 0.
$

We introduce auxiliary coefficients   
$
\tilde{a}_{j}(n),\, \tilde{b}_{j}(n) \in \mathcal{S}\left(2^{j}\right),
j \ge 2,
$
\begin{equation}
\label{coef1}
    \begin{cases}
        \tilde{a}_{j}(n) = \cos t^0_1\, {\rm e}^{i t^0_2}, \\
\tilde{a}_{j}(n+2^{j-1}) = \pm \sin t^0_1\, {\rm e}^{i t^1_2},\\
\tilde{b}_{j}(n) = \sin t^0_1\, {\rm e}^{i t^0_3},\\  
\tilde{b}_{j}(n+2^{j-1}) = \pm \cos t^0_1\, {\rm e}^{ \pm i \left(t^0_3+t^1_2-t^0_2)\right)},\\
    \end{cases}
\end{equation}
where $j \in \mathbb{N}$, $n=0,\dots,2^{j-1}-1,$ $t^0_1,$ $t^0_2,$ $t^0_3,$ and $t^1_2$ 
depend on $j$ and $n$, and signs ``+'' and ``-'' are chosen independently.

Then the collection $\left\{\psi^1_0, \psi^{2}_0\right\}\cup\left\{S^k_j\psi_j : j\in\mathbb{N},  k\in\mathcal{R}_j\right\}$ forms a Parseval wavelet frame in $L_2$
iff the following  conditions are fulfilled
$$
 \sum\limits_{m=1}^{2}\overline{\widehat{b^m_{1}}(0)}\widehat{b^m_{1}}(1) = 0,   
$$
 $$
 \prod_{r=2}^{\infty}  \left|\tilde{a}_{r}(n)\right|^2 = 
2\sum\limits_{m=1}^{2}\left|\widehat{b^m_{1}}(n)\right|^2 \left|\widehat{\varphi_1}(n)\right|^2.      
$$
Moreover,  
$$
\widehat{b_{j}}(n) 
= 
\tilde{b}_{j}(n)
\sqrt{\sum\limits_{m=1}^{2}\left|\widehat{b^m_{1}}(n)\right|^2}
\prod_{r=2}^j  \frac{\left|\widehat{a_{r}}(n)\right|}{    \left|\tilde{a}_{r}(n)\right|}.
$$

\end{corollary}

\begin{proof}
Taking into account Theorem \ref{Th3} all we need is to get a solution of system (\ref{sys2}).  
  It takes the form
$$
\begin{cases}
 \cos t^0_1 \cos t^1_1 \cos\left(t^0_2 - t^1_2\right) + \sin t^0_1 \sin t^1_1 \cos\left(t^0_3 - t^1_3\right) = 0, \\
  \cos t^0_1 \cos t^1_1 \cos\left(t^0_2 - t^1_2\right) + \sin t^0_1 \sin t^1_1 \sin\left(t^0_3 - t^1_3\right) = 0.\\
\end{cases}
$$
Interpreting the equations  as a linear system with respect to $\cos t^0_1 \cos t^1_1$ 
and $\sin t^0_1 \sin t^1_1$, we conclude that if 
$$
\cos\left(t^0_2 - t^1_2\right)\sin\left(t^0_3 - t^1_3\right) -
\cos\left(t^0_3 - t^1_3\right) \cos\left(t^0_2 - t^1_2\right) = 
\sin \left(t^0_3 - t^1_3 -t^0_2 + t^1_2\right)\neq 0,
$$
then $\cos t^0_1 = \sin t^1_1 =0$ or $\cos t^1_1 = \sin t^0_1 =0$, that 
 contradicts the condition $\tilde{a}_{j}(n) \neq 0$.
So, there is  a connection between parameters
$$
t^0_3 - t^1_3 =t^0_2 - t^1_2 +\pi q, \quad q\in \mathbb{Z}.
$$
Therefore,
$$
\begin{cases}
\left(\cos t^0_1 \cos t^1_1 + (-1)^q \sin t^0_1 \sin t^1_1\right) \cos\left(t^0_2 - t^1_2\right) =0,\\
\left(\cos t^0_1 \cos t^1_1 + (-1)^q \sin t^0_1 \sin t^1_1\right) \sin\left(t^0_2 - t^1_2\right) =0,
\end{cases}
$$
or equivalently,
$$
\cos\left(t^0_1 \pm  t^1_1\right) =0.
$$
Thus, general solution of system (\ref{sys0}) in the case of one wavelet generator $\rho_j=1$ can be parameterized by (\ref{coef1}).

\end{proof}
\begin{remark}
    Corollary \ref{Cor1} is a particular case of Theorem \ref{Th3} with $\rho_j=1$ for $j \in \mathbb{N}$ and $\rho_0=2.$ It is impossible to take $\rho_0=1$, otherwise, the assumptions $
\sum\limits_{m=1}^{\rho_{0}}\left|\widehat{b^m_{1}}(n)\right|^2  \neq 0
$   and (\ref{b1}) are inconsistent. 
\end{remark}

The next theorem is the main result of the paper. It gives us the constructive, explicit  characterization of the Parseval wavelet frame without any restrictions on refinable functions.

\begin{theorem}
\label{Th4}
    Suppose  functions $\varphi_j\in L_2,\ j\in\mathbb{N}$, satisfy the refinement 
    equation (\ref{refeq})
for some $\widehat{a_{j+1}}\in\mathcal{S}\left(2^{j+1}\right)$.
For every $j\in\mathbb{Z}_+$ fix a positive integer $\rho_j$ and suppose  $\psi^m_j\in L_2$, $m=1, \ldots, \rho_j$ are defined by the wavelet equation (\ref{waveq})
with some sequences $\widehat{b^m_{j+1}}\in\mathcal{S}\left(2^{j+1}\right),\ m=1, \ldots, \rho_{j}$. Let $\theta_j(n)$ be a sequence  recursively defined by (\ref{theta}). 
 
We introduce auxiliary coefficients   
$
\tilde{a}_{j+1}(n),\, \tilde{b}^m_{j+1}(n) \in \mathcal{S}\left(2^{j+1}\right),
j \in\mathbb{N},
$
by (\ref{tildea}), (\ref{tildeb})
where the parameters 
$t^k_r =t^k_r(j,n),$ $r=1,\dots,  2\rho_j+1,$
satisfy the system (\ref{sys2})
unless $\theta_{j+1}(n+2^{j})\neq 0.$
If $\theta_{j+1}(n+2^{j})= 0$, then the parameters 
$t^k_r =t^k_r(j,n),$ $r=1,\dots,  2\rho_j+1,$
are arbitrary.

Then the collection 
$\left\{S^k_j\psi^m_j : j\in\mathbb{Z}_+, m=1, 2, \ldots, \rho_j, k\in\mathcal{R}_j\right\}$ forms a Parseval wavelet frame in $L_2$
iff the following conditions are fulfilled
		\begin{enumerate}
			\item    $\{j : \widehat{\varphi_j}(n)\neq 0\} 
			\cap
			\left\{
			j : 
			\sum\limits_{m=1}^{\rho_{j-1}}\left|\widehat{b^m_{j}}(n)\right|^2 \neq 0
			\right\}
			\neq \varnothing$;
		\item 
				$ \displaystyle 
				\prod_{r=j_1+1}^{\infty}  \left|\tilde{a}_{r}(n)\right|^2
				=
				2^{j_1} \sum\limits_{m=1}^{\rho_{j_1-1}}\left|\widehat{b^m_{j_1}}(n)\right|^2
				\left|\widehat{\varphi_{j_1}}(n)\right|^2,
			$
   
			where
			$
			j_1:=\min \{j : \widehat{\varphi_j}(n)\neq 0\} 
			\cap
			\left\{
			j : 
			\sum\limits_{m=1}^{\rho_{j-1}}\left|\widehat{b^m_{j}}(n)\right|^2 \neq 0
			\right\}.
			$
			\item $\widehat{b^m_{j}}(n)$ is defined as follows
			$$
\begin{array}{ll}
	(a) \quad \text{arbitrary, }    & j\in \{1,\dots, j_0\},  
	\\
	(b) \quad  0,   &  j\in \{j_0+1,\dots, j_1-1\},
	\\
	(c) \quad    
	\sum\limits_{m=1}^{\rho_{j_1-1}}\overline{\widehat{b^m_{j_1}}(n)}\widehat{b^m_{j_1}}(n+2^{j_1-1}) = 0,
	& 
	j=j_1,\
	\\
	(d) \quad \displaystyle \tilde{b}^m_j(n) \sqrt{\sum\limits_{m=1}^{\rho_{j_1-1}}\left|\widehat{b^m_{j_1}}(n)\right|^2} \prod_{r=j_1+1}^j \frac{\left|\widehat{a_{r}}(n)\right|}{\left|\tilde{a}_{r}(n)\right|},    &   j\in \{j_1+1,j_1+2,\dots\},
\end{array}
$$
where $j_0+1:= \min \{j : \widehat{\varphi_j}(n)\neq 0\}.$
		\end{enumerate}
\end{theorem}

\begin{proof}  

Given $n\in\mathbb{Z}.$
Suppose the collection
$\left\{S^k_j\psi^m_j : j\in\mathbb{Z}_+, m=1, 2, \ldots, \rho_j, k\in\mathcal{R}_j\right\}$ forms a Parseval wavelet frame in $L_2$.
Then by (\ref{con1}),  
$
\{j : \widehat{\varphi_j}(n)\neq 0\} \neq \varnothing. 
$
Therefore, 
by  refinement equation (\ref{refeq}) there exists $j_0=j_0(n)$
such that 
$$
\{j : \widehat{\varphi_j}(n)\neq 0\} = \{j : j > j_0\}.
$$
Then, by (\ref{1th1}) and (\ref{waveq}), 
$$
\sum_{j=0}^{\infty} 2^{j+1} \left( \sum_{m=1}^{\rho_j} 
\left|\widehat{b^m_{j+1}}(n)\right|^2 \right)
\left|\widehat{\varphi_{j+1}}(n)\right|^2 
=
\sum_{j=j_0}^{\infty} 2^{j+1} \left( \sum_{m=1}^{\rho_j} 
\left|\widehat{b^m_{j+1}}(n)\right|^2 \right)
\left|\widehat{\varphi_{j+1}}(n)\right|^2 
= 1.
$$
Therefore, 
$$
 \{j : \widehat{\varphi_j}(n)\neq 0\} 
			\cap
			\left\{
			j : 
			\sum\limits_{m=1}^{\rho_{j-1}}\left|\widehat{b^m_{j}}(n)\right|^2 \neq 0
			\right\}
			\neq \varnothing.
$$
Thus, (1) holds true. 

Now we check that the wavelet masks $\widehat{b^m_{j}}(n)$ satisfy (3).  First of all, we note that 
since $\widehat{\varphi_1}(n) = \dots = \widehat{\varphi_{j_0}}(n) = 0$, it follows that  
$\widehat{a_{1}}(n),$ $\dots,$ $\widehat{a_{j_0}}(n)$, $\widehat{b^m_{1}}(n),$
$\dots,$ $\widehat{b^m_{j_0}}(n)$ can be chosen arbitrary. Thus, (3)(a) follows.  
By 
the wavelet equation 
$\widehat{\psi^m_1}(n)=\dots =\widehat{\psi^m_{j_0}}(n)=0.$ 

For $j\in \{j_0+1,\dots, j_1-1\}$, the definition of $j_1$
implies (3)(b).

To check (3)(c) we note that 
wavelet masks 
$\widehat{b^m_{j}}(n)$
satisfy system (\ref{sys_i}). In particular, the first equation has the form
$$
\theta_{j_1-1}(n)\overline{\widehat{a_{j_1}}(n)}\widehat{a_{j_1}}(n+2^{j_1-1})+\sum\limits_{m=1}^{\rho_{j_1-1}}\overline{\widehat{b^m_{j_1}}(n)}\widehat{b^m_{j_1}}(n+2^{j_1-1}) = 0.
$$
The definition of $j_0$ and  refinement equation (\ref{refeq}) implies $\widehat{a_{j_0+1}}(n)=0.$ Therefore, if  $j_1=j_0+1$, then we immediately get (3)(c). If  $j_1\ge j_0+2$, then   
it follows from (3)(b), recursion (\ref{theta}), and $\widehat{a_{j_0+1}}(n)=0$  that
$$
\theta_{j_1-1}(n) = \theta_{j_0}(n)
\prod_{r=j_0+1}^{j_1-1}\left|\widehat{a_r}(n)\right|^2 = 0. 
$$
Thus, we obtain (3)(c) as well. Actually, we obtain an equivalence of (3)(c) and (\ref{con2}) for $j=j_1.$

For $j\in \{j_1+1, j_1+2, \dots \}$ we again consider system (\ref{sys_i}).
By the definition $j_1$, we have
$
\displaystyle 
\sum\limits_{m=1}^{\rho_{j_1-1}}\left|\widehat{b^m_{j_1}}(n)\right|^2 \neq 0,
$
so by recursion (\ref{theta}) 
$$
   \theta_{j_1}(n) = \sum\limits_{m=1}^{\rho_{j_1-1}}\left|\widehat{b^m_{j_1}}(n)\right|^2 + 
    \theta_{j_1-1}(n) \left|\widehat{a_{j_1}}(n)\right|^2  \neq 0. 
$$
Therefore,  by $\widehat{a_{j}}(n) \neq 0$ and recursion (\ref{theta})
we obtain $\theta_{j}(n) \neq 0$ for $j\ge j_1.$

 There are two options for 
$\widehat{a_{j}}(n+2^{j-1}).$ 
 The 1-st case. $\widehat{a_{j}}(n+2^{j-1})\neq 0.$
In this case 
$\theta_{j}(n+2^{j-1}) \neq 0$ by (\ref{theta}) and $\theta_{j-1}(n) \neq 0$.
That is why we can proceed as in the proof of Theorem \ref{Th3}. Indeed, we introduce auxiliary coefficients 
$\tilde{a}_{j}(n)$, $\tilde{a}_{j}(n+2^{j-1})$, $\tilde{b}^m_{j}(n)$, and 
$\tilde{b}^m_{j}(n+2^{j-1})$ by (\ref{conthet}). We solve system (\ref{sys0}) with respect to these coefficients, obtain solutions (\ref{tildea}), (\ref{tildeb}), and  (\ref{sys2})  and define  $\theta_{j}(n)$ by (\ref{thetarec}) 
and $\widehat{b^m_{j}}(n)$ by (\ref{brec}). Therefore, (3)(d) follows  by recursion.

The 2-nd case. $\widehat{a_{j}}(n+2^{j-1}) = 0.$
If $\theta_{j}(n+2^{j-1}) = 0$, then 
$\widehat{b^m_{j}}(n+2^{j-1}) = 0$ for $m=1,\dots, \rho_{j-1}$. Therefore, 
 (\ref{con2}) is fulfilled. We need to satisfy the only equation 
 (\ref{theta}) 
 $$
    \sum\limits_{m=1}^{\rho_{j-1}}\left|\widehat{b^m_{j}}(n)\right|^2 + 
    \theta_{j-1}(n) \left|\widehat{a_{j}}(n)\right|^2 = \theta_{j}(n).
 $$
 Paramertizing by (\ref{conthet}) we get
 $$
  \sum\limits_{m=1}^{\rho_{j-1}}\left|\tilde{b}^m_{j}(n)\right|^2 + 
    \left|\tilde{a}_{j}(n)\right|^2 =1.
 $$
 The solution is represented in (\ref{tildea}) and (\ref{tildeb}) for $k=0.$ The coefficients  $\theta_{j}(n)$  
and $\widehat{b^m_{j}}(n)$ are defined by (\ref{thetarec}) and (\ref{brec})
respectively. The case (3)(d) follows  by recursion. Note that in this case we do not need  (\ref{sys2}). 

Suppose $\theta_{j}(n+2^{j-1}) \neq 0$, then we can again proceed as in the proof of Theorem \ref{Th3}. The only difference is that  now $\tilde{a}_{j}(n+2^{j-1})=0.$
However, we do not need to find $\widehat{b^m_{j}}(n+2^{j-1})$ at this step. So, we do not need to divide by $\tilde{a}_{j}(n+2^{j-1}).$ Again, (\ref{thetarec}) and (\ref{brec}) define  the coefficients  $\theta_{j}(n)$  
and $\widehat{b^m_{j}}(n)$ respectively. Therefore, (3)(d) follows by recursion.
We note that in fact  we have checked here an equivalence (3)(d) and system (\ref{sys_i}) for  $j\in \{j_1+1, j_1+2, \dots \}$.

Now we check that (\ref{con1}) is equivalent to (2).
Since for $j>j_1$ the coefficients  $\theta_{j}(n)$  
and $\widehat{b^m_{j}}(n)$ are defined by (\ref{thetarec}) and (\ref{brec})
respectively, it is possible to apply recursion. As a result, we get as in the proof
of Theorem \ref{Th3}
$$
 \theta_{j}(n)
=
\theta_{j_1}(n)
\prod_{r=j_1+1}^{j}  \frac{\left|\widehat{a_{r}}(n)\right|^2}{    \left|\tilde{a}_{r}(n)\right|^2}. 
$$
 By (\ref{refeq}), we obtain
$$
\widehat{\varphi_{j_1}}(n) 
= 2^{(j-j_1)/2}\widehat{\varphi_j}(n) \prod_{r=j_1+1}^{j}  \widehat{a_r}(n) .
$$
Therefore,
\begin{equation*}
    \begin{split}
  &
  2^j |\widehat{\varphi_j}(n)|^2 \theta_j(n) 
 = 
  2^j |\widehat{\varphi_j}(n)|^2 
\theta_{j_1}(n)
\prod_{r=j_1+1}^j  \frac{\left|\widehat{a_{r}}(n)\right|^2}{\left|\tilde{a}_{r}(n)\right|^2}
\\
&
=2^{j_1} \left|\widehat{\varphi_{j_1}}(n)\right|^2 
\theta_{j_1}(n)
\frac{1}{\prod_{r=j_1+1}^j  \left|\tilde{a}_{r}(n)\right|^2}.    \end{split}
\end{equation*}
Thus, (\ref{con1}) is equivalent to  (2).

To check sufficiency we examine conditions (\ref{con1})-(\ref{theta}) of Theorem \ref{Th2}. Given $n\in\mathbb{Z}.$
It follows from (1) that there exist $j_0$ and $j_1$ as they are defined above. Then by (\ref{refeq}) 
$\widehat{a_{j_0+1}}(n)=0$ and $\widehat{a_{j}}(n) \neq 0$ for $j\ge j_0+2$ and by (\ref{theta}) $\theta_j(n) =0$ for $j\in \{j_0+1,\dots, j_1-1\}$ and 
$\theta_j(n) \neq 0$ for $j\ge j_1.$ 

We check that (3) implies (\ref{con2}) and (\ref{theta}). For $j \in \{1,\dots,j_0\}$  by the definition of $j_0$ 
we get $\widehat{\varphi_j}(n)=0$, so we do not need to check (\ref{con2}), therefore 
(3)(a) suits us. For $j\in \{j_0+1,\dots, j_1-1\}$ we have 
$\theta_j(n) =0$ and $\widehat{b^m_j}(n)=0,$ $m = 1,\dots, \rho_{j-1},$ therefore 
(\ref{con2}) and (\ref{theta}) hold true. For $j=j_1$ condition (3)(c) implies (\ref{con2}) as it is noted above in the proof of necessity. For $j\ge j_1+1$ condition (3)(d) gives a solution of system (\ref{sys_i}) as it is discussed in the proof of necessity. 
\end{proof}

Among all the conditions representing the sufficiency part of Theorems \ref{Th3}, \ref{Th4}, and Corollary \ref{Cor1} the most complicated seems (\ref{infpr}) and its generalization (2).  In fact, it is not difficult to select the parameters $t^k_r$ defining (\ref{tildea}) to satisfy these conditions. To avoid cumbersomeness and to discuss the tightest with respect to the amount of parameters case, we consider setup of one wavelet generator on each level $j \in \mathbb{N}$, that is $\rho_j=1$.  
 To simplify notations we put $\widehat{b}_j(n)$ instead of $\widehat{b}^1_j(n).$   So, our  question is how to choose the parameters $t^0_1(j,n)$ defining 
 $\tilde{a}_j(n)$ by (\ref{coef1}) to satisfy (2) (or its  particular case (\ref{infpr})). 
First of all, we note that (2) is equivalent to 
\begin{equation}
\label{xi}
  \prod_{r=j_1+1}^{N}  \left|\tilde{a}_{r}(n)\right|^2 = 
\frac{2^{j_1}\left|\widehat{b_{j_1}}(n)\right|^2 \left|\widehat{\varphi_{j_1}}(n)\right|^2}{\xi(n,N)},  
\end{equation}
where 
$\xi(n,N)\to 1$ as $N\to \infty$ for all $n \in \mathbb{Z}.$
To be brief, we denote 
$$
\displaystyle
F(n,N) :=\frac{2^{j_1}\left|\widehat{b_{j_1}}(n)\right|^2 \left|\widehat{\varphi_{j_1}}(n)\right|^2}{\xi(n,N)},
\qquad
\alpha_j(n) := t^0_1(j,n)
$$ 
and consider a sequence of finite dimensional systems 
\begin{equation}
\label{sysN}
    \prod_{r=j_1+1}^{N}  \left|\tilde{a}_{r}(n)\right|^2 =
F(n,N)
\quad n=0,\dots,2^{N-1}-1.
\end{equation}
By periodicity of $\tilde{a}_j$ and (\ref{coef1}) for the remaining $n$, the function $F(n,N)$ satisfy relations 
$$
F(n-2^{N-1},N)+F(n,N) = F(n-2^{N-1},N-1) \mbox{ for } n=2^{N-1},\dots,2^{N}-1,
$$
and $F(n,N)$ is $2^N$-periodic with respect to the first variable $n.$ These properties are satisfied by choice of $\xi(n,N)$. We consider two examples.

\begin{example}
\label{Ex1}
Let $\widehat{\varphi_j}(n) \neq 0$ for all $j\in \mathbb{N}$, $n\in\mathbb{Z}$
and 
$
\sum\limits_{m=1}^{2}\left|\widehat{b^m_{1}}(n)\right|^2 \neq 0.
$
So, we are in assumptions of Corollary \ref{Cor1}. In this case $j_0(n)=0$, $j_1(n)=j_0(n)+1=1$ for all $n\in\mathbb{Z}.$ 
Taking into account periodicity of $\tilde{a}_j$ and (\ref{coef1}) we get for $N=2,$ $N=3,$ $N=4$
$$
\begin{cases}
    \left|\cos \alpha_2(0)\right|^2 = F(0,2), \\
    \left|\cos \alpha_2(1)\right|^2 = F(1,2),
\end{cases}
\begin{cases}
    \left|\cos \alpha_2(0)\right|^2 \left|\cos \alpha_3(0)\right|^2 = F(0,3), \\
    \left|\cos \alpha_2(1)\right|^2 \left|\cos \alpha_3(1)\right|^2  = F(1,3), \\
     \left|\sin \alpha_2(0)\right|^2 \left|\cos \alpha_3(2)\right|^2 = F(2,3), \\
    \left|\sin \alpha_2(1)\right|^2 \left|\cos \alpha_3(3)\right|^2  = F(3,3), \\
\end{cases}
$$
$$
\begin{cases}
    \left|\cos \alpha_2(0)\right|^2 \left|\cos \alpha_3(0)\right|^2 \left|\cos \alpha_4(0)\right|^2= F(0,4), \\
    \left|\cos \alpha_2(1)\right|^2 \left|\cos \alpha_3(1)\right|^2\left|\cos \alpha_4(1)\right|^2  = F(1,4), \\
     \left|\sin \alpha_2(0)\right|^2 \left|\cos \alpha_3(2)\right|^2 \left|\cos \alpha_4(2)\right|^2= F(2,4), \\
    \left|\sin \alpha_2(1)\right|^2 \left|\cos \alpha_3(3)\right|^2  \left|\cos \alpha_4(3)\right|^2= F(3,4), \\
    \left|\cos \alpha_2(0)\right|^2 \left|\sin \alpha_3(0)\right|^2 \left|\cos \alpha_4(4)\right|^2= F(4,4), \\
    \left|\cos \alpha_2(1)\right|^2 \left|\sin \alpha_3(1)\right|^2  \left|\cos \alpha_4(5)\right|^2= F(5,4), \\
     \left|\sin \alpha_2(0)\right|^2 \left|\sin \alpha_3(2)\right|^2 \left|\cos \alpha_4(6)\right|^2= F(6,4), \\
    \left|\sin \alpha_2(1)\right|^2 \left|\sin \alpha_3(3)\right|^2  \left|\cos \alpha_4(7)\right|^2= F(7,4). \\
\end{cases}
$$
In general case, to form the $N+1$-th system we take the left-hand side of the previous one and write it in the first $2^N$ equations, then we repeat these expressions replacing     
$\left|\cos \alpha_N(n)\right|^2$ by $\left|\sin \alpha_N(n)\right|^2$ and put them into the second $2^N$ equations, finally we multiply each expression by $\left|\cos \alpha_{N+1}(n)\right|^2.$
All systems can be solved consequently. We find
$\left|\cos \alpha_{N}(n)\right|^2$
from $N$-th system. For the first three steps we obtain 
$$
\begin{cases}
    \left|\cos \alpha_2(0)\right|^2 = F(0,2), \\
    \left|\cos \alpha_2(1)\right|^2 = F(1,2),
\end{cases}
\begin{cases}
   \left|\cos \alpha_3(0)\right|^2 = \frac{F(0,3)}{F(0,2)}, \\
     \left|\cos \alpha_3(1)\right|^2  = \frac{F(1,3)}{F(1,2)}, \\
     \left|\cos \alpha_3(2)\right|^2 = \frac{F(2,3)}{1- \left|\cos \alpha_2(0)\right|^2}, \\
     \left|\cos \alpha_3(3)\right|^2  = \frac{F(3,3)}{1-\left|\cos \alpha_2(1)\right|^2}, \\
\end{cases}
$$
$$
\begin{cases}
    \left|\cos \alpha_4(0)\right|^2= \frac{F(0,4)}{F(0,3)}, \\
    \left|\cos \alpha_4(1)\right|^2  =\frac{F(1,4)}{F(1,3)}, \\
      \left|\cos \alpha_4(2)\right|^2= \frac{F(2,4)}{F(2,3)}, \\
      \left|\cos \alpha_4(3)\right|^2= \frac{F(3,4)}{F(3,3)}, \\
      \left|\cos \alpha_4(4)\right|^2= \frac{F(4,4)}{F(0,2)\left(1-\left|\cos \alpha_3(0)\right|^2\right)}, \\
      \left|\cos \alpha_4(5)\right|^2= \frac{F(5,4)}{F(1,2)\left(1-\left|\cos \alpha_3(1)\right|^2\right)}, \\
   \left|\cos \alpha_4(6)\right|^2= \frac{F(6,4)}{\left(1-\left|\cos \alpha_2(0)\right|^2\right)\left(1-\left|\cos \alpha_3(2)\right|^2\right)}, \\
   \left|\cos \alpha_4(7)\right|^2= \frac{F(7,4)}{\left(1-\left|\cos \alpha_2(1)\right|^2\right)\left(1-\left|\cos \alpha_3(3)\right|^2\right)}. \\
\end{cases}
$$
In general case we get the following solution.

If $n=0,\dots,2^{N-2}-1$, then
$$
   \left|\cos \alpha_N(n)\right|^2 = \frac{F(n,N)}{F(n,N-1)}.
$$
If $n=2^{N-2}+k,\, k=0,\dots,+2^{N-3}-1$, then
\begin{equation*}
    \begin{split}
  &
  \left|\cos \alpha_N(n)\right|^2 = \frac{F(2^{N-2}+k,N)}{F(k,N-2)\left(1-\left|\cos \alpha_{N-1}(k)\right|^2\right)}
\\
&
=\frac{F(2^{N-2}+k,N)}{F(k,N-2)-F(k,N-1)}.
    \end{split}
\end{equation*}
If $n=2^{N-2}+2^{N-3}+k, \,k=0,\dots,2^{N-2}+2^{N-3}+2^{N-4}-1$, then
\begin{equation*}
    \begin{split}
  &
   \left|\cos \alpha_N(n)\right|^2 = \frac{F(2^{N-2}+2^{N-3}+k,N)}{F(k,N-3)\left(1-\left|\cos \alpha_{N-2}(k)\right|^2\right)\left(1-\left|\cos \alpha_{N-1}(2^{N-3}+k)\right|^2\right)}
\\
&
=\frac{F(2^{N-2}+2^{N-3}+k,N)}{F(k,N-3)-F(k,N-2)-F(2^{N-3}+k,N-1)}.
\end{split}
\end{equation*}
If $n=2^{N-2}+2^{N-3}+2^{N-4}+k, \,k=0\dots,2^{N-2}+2^{N-3}+2^{N-4}+2^{N-5}-1$, 
then
\begin{equation*}
    \begin{split}
  &
   \left|\cos \alpha_N(n)\right|^2 
=F(2^{N-2}+2^{N-3}+2^{N-4}+k,N)\left(F(k,N-4) \right.
\\
&
\left.
-F(k,N-3)
-F(2^{N-4}+k,N-2)-F(2^{N-3}+2^{N-4}+k,N-1)
\right)^{-1}.
\end{split}
\end{equation*}
If $n=2^{N-2} + 2^{N-3} + \dots +2^1+k,\, k=0,1$ , then
\begin{equation*}
    \begin{split}
  &
   \left|\cos \alpha_N(n)\right|^2 
=F(2^{N-2} + 2^{N-3} + \dots +2^1+k,N)\left(1- F(k,2)\right.
\\
&
\left.
-F(2+k,3)
-\dots-F(2^{N-3}+\dots+2+k,N-1)
\right)^{-1}.
\end{split}
\end{equation*}
Since $F(n,N)>0,$ it follows that $\left|\cos \alpha_N(n)\right|,$
$\left|\sin \alpha_N(n)\right|\neq 0$. So, the expressions on the right-hand side are well-defined. However, there is another natural restriction, namely these expressions should be in the interval $(0,1)$. Since all the nominators are strictly positive, the last requirement is equivalent to the fact that the difference between the denominator and the nominator is strictly positive. Comparing the solutions of $N$-th and $N+1$-th system we notice that $r+1$-th inequality of $N+1$-th system implies $r$-th inequality of $N$-th system. For example, the third inequality of   $N+1$-th system
$$
F(k,N-2)-F(k,N-1)-F(2^{N-2}+k,N)-F(2^{N-1}+2^{N-2}+k,N+1)>0,
$$
$n=2^{N-1}+2^{N-2}+k, \,k=0,\dots,+2^{N-3}-1$,
implies the second inequality of $N$-th system
$$
F(k,N-2)-F(k,N-1)-F(k+2^{N-2},N)>0
$$
for $n=2^{N-2}+k,\, k=0,\dots,+2^{N-3}-1$.
Therefore, all the inequalities of the fixed chain of inequalities
hold simultaneously iff we extend the sum to an infinite number of terms.
In the above example the  chain consists of $p$-th inequalities of $N+p-2$-th systems, $p\in\mathbb{N}$, and the resulting inequality takes the form  
$$
F(k,N-1)+F(2^{N-2}+k,N)+F(2^{N-1}+2^{N-2}+k,N+1)+\dots<F(k,N-2),
$$
$k=0,\dots,+2^{N-3}-1.$
For all the chains we obtain 
\begin{equation}
\label{solinfpr1}
F(k,2)+\sum_{N=1}^{\infty}F(k+2^1+\dots+2^N, N+2)<1, \quad k=0,1,    
\end{equation}
\begin{equation}
\label{solinfpr2}
F(k,M+1)+\sum_{N=M}^{\infty}F(k+2^M+\dots+2^N, N+2)<F(k,M), 
\end{equation}
$k=0,\dots,2^{M-1}-1, M\ge 2.$
\end{example}
\begin{remark}
    Let us summarize findings of Example \ref{Ex1}. Suppose we take a sequence of non-zero numbers  $\left(\widehat{\varphi_1}(n)\right) \in l_2$, $n\in \mathbb{Z},$ and four numbers $\widehat{b^m_{1}}(n),$ $m=1,2$,
    $n=0,1$ such that 
    $
\sum\limits_{m=1}^{2}\left|\widehat{b^m_{1}}(n)\right|^2 \neq 0
$
and
$
 \sum\limits_{m=1}^{2}\overline{\widehat{b^m_{1}}(0)}\widehat{b^m_{1}}(1) = 0.   
$ 
Next, we take arbitrary non-zero coefficients $\widehat{a_{j+1}}\in\mathcal{S}\left(2^{j+1}\right)$ and define  $\widehat{\varphi_{j+1}}(n),$ $j \in \mathbb{N},$ $n\in\mathbb{Z}.$
Then refinable functions $\varphi_{j}$, $j \in \mathbb{N},$
generate a Parseval wavelet frame iff
 inequalities 
(\ref{solinfpr1}) and  (\ref{solinfpr2}) are fulfilled. Thus, the existence of the frame 
completely depends on the sequence $\widehat{\varphi_1}(n)$, $n\in \mathbb{Z}.$
\end{remark}
\begin{example}
\label{Ex2}
    We still consider a case of one wavelet generator,  $\rho_j=1$ for $j\in\mathbb{N}$, $\rho_0=2$.  Let refinable functions  $\varphi_{j}$ be trigonometric polynomials   and $\{n\,:\, \widehat{\varphi_{j}}(n) \neq 0\}=\{-2^j,\dots, 2^j-1\}.$
Then by definitions of $j_0$  we get
$j_0(n)+1 = \left\lceil\log_2 (n+1) \right\rceil$ for $n\in\mathbb{N}$ 
and $j_0(0)+1 =1.$ Let $j_1(n)=j_0(n)+1,$ that is $\widehat{b}_{j_0(n)+1}(n)\neq 0.$
By (3)(c) we need equality 
$
\overline{\widehat{b}_{j_0(n)+1}(n)} \widehat{b}_{j_0(n)+1}(n+2^{j_0(n)}) = 0.
$
It is straightforward to see that $j_0(n+2^{j_0(n)})> j_0(n)$ for $n\in\mathbb{N}$, so by (3)(a) the coefficient
$\widehat{b}_{j_0(n)+1}(n+2^{j_0(n)})$ can be chosen arbitrary, we put 
$\widehat{b}_{j_0(n)+1}(n+2^{j_0(n)}) =0$
to satisfy (3)(c). 
Taking into account  the definition of $j_1(n)$ and (\ref{coef1}) system (\ref{sysN}) has the form for $N=2,$ $N=3,$ $N=4$
$$
\begin{cases}
    \left|\cos \alpha_2(0)\right|^2 = F(0,2), \\
    \left|\cos \alpha_2(1)\right|^2 = F(1,2),
\end{cases}
\begin{cases}
    \left|\cos \alpha_2(0)\right|^2 \left|\cos \alpha_3(0)\right|^2 = F(0,3), \\
    \left|\cos \alpha_2(1)\right|^2 \left|\cos \alpha_3(1)\right|^2  = F(1,3), \\
     \left|\cos \alpha_3(2)\right|^2 = F(2,3), \\
     \left|\cos \alpha_3(3)\right|^2  = F(3,3), \\
\end{cases}
$$
$$
\begin{cases}
    \left|\cos \alpha_2(0)\right|^2 \left|\cos \alpha_3(0)\right|^2 \left|\cos \alpha_4(0)\right|^2= F(0,4), \\
    \left|\cos \alpha_2(1)\right|^2 \left|\cos \alpha_3(1)\right|^2\left|\cos \alpha_4(1)\right|^2  = F(1,4), \\
      \left|\cos \alpha_3(2)\right|^2 \left|\cos \alpha_4(2)\right|^2= F(2,4), \\
    \left|\cos \alpha_3(3)\right|^2  \left|\cos \alpha_4(3)\right|^2= F(3,4), \\
    \left|\cos \alpha_4(4)\right|^2= F(4,4), \\
      \left|\cos \alpha_4(5)\right|^2= F(5,4), \\
    \left|\cos \alpha_4(6)\right|^2= F(6,4), \\
      \left|\cos \alpha_4(7)\right|^2= F(7,4). \\
\end{cases}
$$
In general case, to form the $N+1$-th system we take the left-hand side of the previous one and write it in the first $2^N$ equations, then we put just $1$'s into the second $2^N$ equations, finally we multiply each expression by $\left|\cos \alpha_{N+1}(n)\right|^2.$ These systems are much simpler than systems of Example \ref{Ex1}. We immediately get the solution
$$
\left|\cos \alpha_N(n)\right|^2= 
\begin{cases}
\displaystyle 
    \frac{F(n,N)}{F(n,N-1)} = \frac{\xi(n,N-1)}{\xi(n,N)}, \quad n=0,\dots,2^{N-2}-1, \\
 \displaystyle    F(n,N) = \frac{2^{j_1}\left|\widehat{b_{j_1}}(n)\right|^2 \left|\widehat{\varphi_{j_1}}(n)\right|^2}{\xi(n,N)}, \quad n=2^{N-2},\dots,2^{N-1}-1.
\end{cases}
$$
The solution is correct iff
$0< \left|\cos \alpha_N(n)\right|^2< 1$, the last condition can be controlled by the choice of $\xi$. Namely, 
\begin{equation*}
    \begin{split}
      &
      \xi(n,N-1) < \xi(n,N) , \quad n=0,\dots,2^{N-2}-1,
\\
&
\xi(n,N) > 2^{j_1}\left|\widehat{b_{j_1}}(n)\right|^2 \left|\widehat{\varphi_{j_1}}(n)\right|^2, \quad n=2^{N-2},\dots,2^{N-1}-1.
    \end{split}
\end{equation*}
\end{example}

\end{document}